\documentclass[11pt]{amsart}
\usepackage{amsmath,amssymb,amsthm}
\usepackage{float}
\usepackage{tikz-cd}
\usepackage[shortlabels]{enumitem}

\newtheorem{theorem}{Theorem}
\newtheorem*{theorem*}{Theorem}
\newtheorem{proposition}[theorem]{Proposition}

\newtheorem{lemma}[theorem]{Lemma}

\newtheorem{corollary}[theorem]{Corollary}
\newtheorem*{corollary*}{Corollary}

\theoremstyle{definition}
\newtheorem{definition}[theorem]{Definition}
\newtheorem{remark}[theorem]{Remark}
\newtheorem{example}[theorem]{Example}

\numberwithin{theorem}{section}
\numberwithin{equation}{section}
\numberwithin{figure}{section}

\makeatletter
\def\thmhead@plain#1#2#3{%
  \thmname{#1}\thmnumber{\@ifnotempty{#1}{ }\@upn{#2}}%
  \thmnote{ {\the\thm@notefont#3}}}
\let\thmhead\thmhead@plain
\makeatother

\def\PP{\mathbb{P}}

\def\R{\mathbb{R}}

\def\sC{\mathcal{C}}
\def\sD{\mathcal{D}}
\def\sE{\mathcal{E}}

\def\sL{\mathcal{L}}

\def\sW{\mathcal{W}}

\def\fX{\mathfrak{X}}
\def\fg{\mathfrak{g}}

\def\eps{\epsilon}

\def\bra{\langle}
\def\ket{\rangle}

\newcommand{\spacedt}[1]
{
\quad{\textup{#1}}\quad
}

\newcommand{\Matrix}[1]
{
\left(
\begin{matrix}
#1
\end{matrix}
\right)
}

\newcommand{\restrto}[2]
{
\left.{#1}\right|_{#2}
}

\newcommand{\su}[1]
{
_{_{#1}}
}

\def\imm{\hookrightarrow}

\newcommand{\D}[3][{}]
{
\frac{d^{#1}{#2}}{d{#3}^{#1}}
}

\def\Op{\mathcal{O}p}

\newcommand{\pf}[1]{\mathop{\mathrm{#1}}\nolimits}

\newcommand{\rot}[1]{\pf{rot}_{\gamma,p}\left(#1\right)}

\newcommand{\maxrot}[1][{L}]{\pf{max}\rot{#1}}

\begin{document}

\title{On the dynamics of some vector fields tangent to non-integrable plane fields}


\author{Nicola Pia}

\date{The author is supported by the DAAD programme Research Grants for Doctoral Candidates and Young Academics and Scientists (more than 6 months) No. 57381410, 2018/19}

\address{Mathematisches Institut, LMU M\"unchen, Theresienstr. 39, 80333 M\"unchen, Germany}

\email{pia@math.lmu.de}

\begin{abstract}
Let $\sE^3\subset TM^n$ be a smooth $3$-distribution on a smooth manifold of dimension $n$ and $\sW\subset\sE$ a line field such that $[\sW,\sE]\subset\sE$.
Under some orientability hypothesis, we give a necessary condition for the existence of a plane field $\sD^2$ such that $\sW\subset\sD$ and $[\sD,\sD]=\sE$.
Moreover we study the case where a section of $\sW$ is non-singular Morse-Smale and we get a sufficient condition for the global existence of $\sD$.

As a corollary we get conditions for a non-singular vector field $W$ on a $3$-manifold to be Legendrian for a contact structure $\sD$.
Similarly with these techniques we can study when an even contact structure $\sE\subset TM^4$ is induced by an Engel structure $\sD$.
\end{abstract}

\maketitle

\section{Introduction}
The only topologically stable families of smooth distributions on smooth manifolds are line fields, contact structures, even contact structures, and Engel structure \cite{cartan, engel, gersch, presas}.
An even contact structure is a maximally non-integrable hyperplane field on an even dimensional manifold.
An Engel structure is $2$-plane field $\sD$ on a $4$-manifold $M$ such that $\sE=[\sD,\sD]$ is an even contact structure.
Engel structures were discovered more than a century ago \cite{cartan,engel} and they sparked big interest throughout the years \cite{cp3,montgomeryPaper,vogelGeometricEngel,vogelExistence}.

The motivation behind this paper was to understand which even contact structures $(M^4,\sE)$ are induced by Engel structures $\sD$, i.e. $[\sD,\sD]=\sE$.
There are some obvious topological obstructions since $M$ admits an even contact structure if (up to a $2$-cover) its Euler characteristic vanishes (see \cite{mcduff}), whereas it admits an Engel structure only if it is parallelizable (up to a $4$-cover, see \cite{vogelExistence}).
For this reason we only consider even contact structures $\sE$ which admit a framing $\sE=\bra W,\, A,\, B\ket$ where $W$ spans the characteristic foliation, i.e. the unique line field $\sW\subset\sE$ satisfying $[\sW,\sE]\subset\sE$.
Then an orientable Engel structure compatible with $\sE$ takes the form $\sD_L=\bra W,\,L\ket$, where $L\in\bra A,\, B\ket$ and $[\sD_L,\sD_L]=\sE$.

It turns out that the same framework can be used to describe different contexts.
For example if $M$ is an orientable manifold of dimension $3$ and $\sE:=TM=\bra W,\, A,\, B\ket$, then a plane field of the form $\sD_L=\bra W,\,L\ket$, where $L\in\bra A,\, B\ket$ and $[\sD_L,\sD_L]=\sE$, is an orientable contact structure for which $W$ is Legendrian.
For this reason we can tackle the problem of finding a contact structure such that a given non-singular vector field $W$ is Legendrian with the same tools.
This question has already been studied in the case of Morse-Smale gradient vector fields in \cite{etnyreghrist}.

We introduce a more general family of distributions which permits to treat the above cases at once.
For a given $3$-distribution $\sE\subset TM$ on a manifold $M$ we say that a $2$-plane field $\sD\subset\sE$ \emph{generates} $\sE$ or that $\sD$ is \emph{maximally non-integrable within $\sE$} if $[\sD,\sD]=\sE$.
Moreover if $\sE=\bra W,\, A,\, B\ket$ and $\sW=\bra W\ket$ satisfies $[\sW,\sE]\subset\sE$, we say that $\sD=\bra W,\, L\ket$ generates $\sE$ up to homotopy if there is a homotopy of the form $\sD_s=\bra W,\, L_s\ket$ for $s\in[0,1]$ such that $\sD_1$ generates $\sE$ (see Section~\ref{SEC_Notation} for more details).

Since the flow $\phi_t$ of $W$ preserves $\sE$, for a given $L\in\bra A,\, B\ket$ we can write
\[
 \Big(T_{\phi_{_{t}}(p)}\phi_{-t}\Big)L(\phi_{_{t}}(p))=\rho(p;t)\Big(\cos\theta(p;t)\ A(p)+\sin\theta(p;t)\ B(p)\Big)+c(p;t)\ W(p).
\]
If the function $\theta(p;t)$ has non-vanishing derivative then $\sD=\bra W,\, L\ket$ generates $\sE$ on the orbit of $p$.
For this reason if $\gamma$ is a closed orbit of $W$ through $p$ of period $T$, we consider the quantity
\[
 \rot{L}=\theta(p;T)-\theta(p;0),
\]
which we call the \emph{rotation number of $L$ along $\gamma$ at $p$}.

It turns out that this is not invariant under homotopies of $L$ (see Example~\ref{EXM_RotNotInvariant}).
One needs to consider instead the rotation number of the vector fields obtained by rotating $L$ in the plane $\bra A,\, B\ket$ by a constant phase $\eta$ and take the maximum, this is what we denote by $\maxrot$.
This quantity is invariant under homotopies and it gives an obstruction to the existence of $\sD$ generating $\sE$.
\begin{theorem*}[A]
	Let $\sE=\bra W,\,A,\,B\ket$ be a distribution of rank $3$ such that $[\sW,\sE]\subset\sE$, and let $\gamma$ a closed orbit of $W$.
	Then $\sD_L=\bra W,\,L\ket$ generates $\sE$ on $\gamma$ up to homotopy within $\sE$ if and only if there exists a point $p\in\gamma$ such that $|\maxrot|>0$.
\end{theorem*}
The behaviour of the rotation number under homotopies depends on the type of the closed orbit $\gamma$, i.e. the type of the action of the flow on $\sE/\sW$.
If $\gamma$ is elliptic, the rotation number is invariant under homotopies of $L$.
In this case there exists a vector field $L_1$ homotopic to $L$ through sections of $\sE$ and such that $\bra W,\, L_1\ket$ positively generates $\sE$ on a neighbourhood of $\gamma$ if and only if $\rot{L}>0$.
This condition is very easy to verify if $\gamma$ bounds an embedded disc.

If the dynamics of the vector field $W$ are particularly simple, namely if it is a non-singular Morse-Smale vector field (briefly NMS), we can give a necessary and sufficient condition for the existence of $\sD$ generating $\sE$.

\begin{theorem*}[B]
    Let $\sE=\bra W,\,A,\,B\ket$ be a rank $3$ distribution such that $[\sW,\sE]\subset\sE$, and let $W$ be a NMS vector field.
    There exists $\sD\subset\sE$ that positively generates $\sE$ if and only if there exists a vector field $L\in\bra A,\,B\ket$ such that $\maxrot>0$ for all $\gamma$ closed orbit of $W$.
\end{theorem*}

If we specialize this result in the case of contact structures and Engel structures we obtain the following interesting facts. 

\begin{corollary*}[C]
    Let $M$ be a closed orientable $3$-manifold and let $W$ be a NMS vector field on $M$ such that $TM=\bra W,\,A,\,B\ket$.
    There exists a positive contact structure $\sD$ for which $W$ is Legendrian if and only if there exists a vector field $L\in\bra A,\,B\ket$ such that $\maxrot>0$ for all $\gamma$ closed orbit of $W$.
\end{corollary*}
The previous corollary is new since it applies to non-singular vector fields, whereas the ones treated in \cite{etnyreghrist} always admit singularities.

\begin{corollary*}[D]
 	Let $(M,\,\sE=\bra W,\,A,\,B\ket)$ be a closed, oriented even contact $4$-manifold with Morse-Smale characteristic foliation.
	Then there exists a positive Engel structure $\sD$ compatible with $\sE$ if and only if there exists a vector field $L\in\bra A,B\ket$ such that $\maxrot>0$ for all $\gamma$ closed characteristic orbits.
\end{corollary*}

\subsection{Structure of the paper}
In Section~\ref{SEC_BasicNotions} we recall basic facts of the theory of Engel structures and even contact structures.
Moreover we introduce the concept of maximally non-integrable $2$-plane field within a $3$-distribution.
In Section~\ref{SEC_Notation} we introduce the family of plane fields and homotopies we are interested in, and we relate non-integrability of $\bra W,\, L\ket$ to how $L$ ``rotates around the flow of $W$''.

In Section~\ref{SEC_RotationNumber} we consider points $p\in M$ contained in a closed orbit $\gamma$ of $\sW$.
We introduce the notion of rotation number along $\gamma$ and of maximal rotation number.
We point out how this is related to the existence of a plane field $\sD$ which generates $\sE$.
The behaviour of the rotation number under homotopies depends on the type of the closed orbit $\gamma$.
This is the content of Section~\ref{SEC_CharacterOfWOrbits}.

In Section~\ref{SEC_MSVF} we apply the theory to Morse-Smale vector fields.
These induce a round-handle decomposition of $M$ that we can use to find a necessary and sufficient condition for $\sE$ to admit $\sD$ generating it in the case where $W$ is NMS .
In Sections~\ref{SEC_MSLegendrian} and~\ref{SEC_MSEngel} we specialize the results obtained to the case of Legendrian vector fields on $3$-manifolds and Morse-Smale even contact structures.

\textbf{Acknowledgements:}
I would like to thank my advisors Prof. Gianluca Bande and Prof. Dieter Kotschick.
The intuition driving the definition of the rotation number was pointed out to me during the problem sessions at the AIM workshop \emph{Engel structures} in San Jos\'e, April 2017.
For this I am particularly thankful to Prof. Yakov Eliashberg.
Finally I thank Prof. Thomas Vogel and Prof. Vincent Colin for their useful feedback, Rui Coelho, and Giovanni Placini for having the patience to listen to me all the time.

\section{Basic notions}\label{SEC_BasicNotions}
In what follows all manifolds are closed and smooth, and all distributions are smooth, if not otherwise stated.

An \emph{even contact structure} $\sE$ is a maximally non-integrable hyperplane distribution on an even dimensional manifold.
Otherwise said $\dim M=2n+2$ and locally $\sE$ is the kernel of a 1-form $\sE=\ker\alpha$ satisfying $\alpha\wedge d\alpha^{2n}\ne0$.
For dimensional reasons if $\sE$ is even contact then there exists a unique line field $\sW$ such that $\sW\subset\sE$ and $[\sW,\sE]\subset\sE$.
We call $\sW$ the \emph{characteristic foliation} of $\sE$.
The existence of this line field implies that if $M$ admits an even contact structure then $\chi(M)=0$.
In \cite{mcduff} a complete h-principle for even contact structures is proved.

An \emph{Engel structure} $\sD$ is a $2$-plane field on a $4$-manifold $M$ such that $\sE:=[\sD,\sD]$ is an even contact structure.
One can see that the characteristic foliation $\sW$ of $\sE=[\sD,\sD]$ must satisfy $\sW\subset\sD$.
The flag of distributions $\sW\subset\sD\subset\sE$ is called the \emph{Engel flag of} $\sD$, and its existence gives strong constraints on the topology of the manifold $M$.
Indeed maximal non-integrability ensures that
\begin{equation}\label{EQN_IsoOfDetBundles}
	\det(\sE/\sW)\cong\det(TM/\sE)\qquad\textup{and}\qquad\det(\sE/\sD)\cong\det(\sD),
\end{equation}
this and the existence of the Engel flag in turn imply that $M$ admits a parallelizable $4$-cover.
The study of the space of Engel structures on parallelizable $4$-manifolds is the subject of fruitful research \cite{cp3,overtwistedEngel,vogelExistence}.

Let $\sE$ be an even contact structures on a $4$-manifold $M$, we say that an Engel structure $\sD$ on $M$ is \emph{compatible} with $\sE$ if $[\sD,\sD]=\sE$.
Otherwise said $\sD$ is compatible with $\sE$ if the latter is its induced even contact structure.
The motivation behind this paper was to understand when, for a given even contact structure $(M,\sE)$ on $4$-manifold, there exists a compatible Engel structure.
There are some obvious topological obstructions to the existence of a compatible Engel structure given by the fact that $M$ must admit a parallelizable $4$-cover.
In order to avoid these cases we suppose that $\sE$ is trivial as a bundle and admits a global framing $\sE=\bra W,\, A,\, B\ket$, where $W$ spans the characteristic foliation $\sW$.
Equation~\eqref{EQN_IsoOfDetBundles} ensures that $M$ is orientable, so that $\sE$ is co-orientable and $TM$ is trivial.
Now there exists an orientable Engel structure $\sD$ compatible with $\sE$ if and only if there exists a vector field $L\in\bra A,\, B\ket$ such that $\sD=\bra W,\,L\ket$ satisfies $[\sD,\sD]=\sE$.

Notice that the same formalism can be used to describe orientable contact structures on $3$-manifolds which are trivial as bundles.
Namely let $M$ be an orientable $3$-manifold and suppose that $W$ is a non-singular vector field such that $TM=\bra W,\, A,\, B\ket$.
Then $W$ is tangent to an orientable contact structure $\sD$ if and only if there exists a vector field $L\in\bra A,\, B\ket$ such that $\sD=\bra W,\,L\ket$ satisfies $[\sD,\sD]=TM$.

In order to treat all these cases at once we give the following
\begin{definition}
 Let $M^n$ be a manifold of dimension $n$, and let $\sE\subset TM$ be a rank $3$ distribution.
 We say that a $2$-plane field $\sD\subset\sE$ \emph{generates} $\sE$ or that $\sD$ is \emph{maximally non-integrable within $\sE$} if $[\sD,\sD]=\sE$.
\end{definition}

If $\sE$ is an even contact structure as above, then Engel structures $\sD$ are maximally non-integrable plane fields within $\sE$.
Similarly if $\sE=TM^3$ as above, then contact structures are maximally non-integrable plane fields within $TM$.
Contact foliations provide another example, indeed suppose that $\sE\subset TM^4$ is an integrable hyperplane distribution on a $4$-manifold, then if $\sD\subset\sE$ generates $\sE$ the leaves of the foliation induced by $\sE$ are (possibly open) contact manifolds.

\section{Rotate without stopping}\label{SEC_Notation}
Suppose that $\sE\subset TM$ is a rank $3$ distribution which admits a global framing $\sE=\bra W,\, A,\, B\ket$ such that the flow of $W$ preserves $\sE$.
Moreover denote by $\sW$ the line field spanned by $W$.
Let $L\in\Gamma\sE$ be never tangent to $\sW$, we want to determine when the distribution $\sD_L:=\bra W,\,L\ket$ is homotopic within $\sE$ to maximally non-integrable plane field within $\sE$.
Notice that, since we have fixed a framing, $\sE$ is oriented, and every $\sD$ that generates $\sE$ also defines and orientation of $\sE$.
\begin{definition}
	Suppose that $\sE=\bra W,\, A,\, B\ket\subset TM$ such that the flow of $W$ preserves $\sE$, let $L\in\Gamma\sE$ be never tangent to $\sW$, and $K\subset M$.
	We say that $\sD_L=\bra W,\,L\ket$ \emph{generates} $\sE$ \emph{on} $K$ \emph{up to homotopy within} $\sE$ if there exists a smooth family of vector fields $L_\tau\in\Gamma\sE$ for $\tau\in[0,1]$ such that $L_\tau$ is never tangent to $W$ for all $\tau\in[0,1]$, $L_0=L$ and $\sD_{L_1}$ generates $\sE$ on a neighbourhood of $K$.
	
	Moreover we say that $\sD_L$ \emph{positively (resp. negatively) generates} $\sE$ \emph{on} $K$ \emph{up to homotopy within} $\sE$ if $\sD_{L_1}$ induces the orientation of $\sE$ (resp. induces the opposite orientation).
\end{definition}

\begin{remark}\label{REM_Homotopies}
 Notice that we do not consider all possible homotopies of the plane field $\sD\subset\sE$ but only those tangent to $\sW$.
 This will be enough in order to answer the question about existence of $\sD$.
\end{remark}

Notice that the framing $\{A,\,B\}$ allows us to identify $L$ with a map $L:M\to S^1$.
It will be useful sometimes to lift this map to the universal cover $\tilde M$ so that the vector field takes the form $L=\cos\psi\ A+\sin\psi\ B$ for some function $\psi:\tilde M\to\R$ making the following diagram commute

\[
\begin{tikzcd}
\tilde M \arrow{r}{\psi} \arrow[swap]{d}{} & \R\ \arrow{d}{} \\%
M \arrow{r}{L}& S^1.
\end{tikzcd}
\]

Consider the $2$-form $\omega$ on $\sE$ given by $\omega(A,B)=1$ and $i_W\omega=0$.
If we define $J\in\pf{End}\sE$ by $JW=0,\ JA=B$ and $JB=-A$, we get
\[
\omega(JX,JY)=\omega(X,Y) 
\]
for all $X,\,Y$ sections of $\sE$ and $\omega(-,J-)$ is positive definite on $\bra A,\,B\ket$.

Denote the flow of $W$ at time $t$ by $\phi_t$, and its tangent map at $p\in M$ by $T_p\phi_t:T_pM\to T_{\phi_t(p)}M$.
Since the flow of $W$ preserves $\sE$ and $i_W\omega=0$, there is a non-vanishing function $\lambda$ which satisfies
\begin{equation}\label{EQN_alphaUnderTheFlow}
(\phi_s^*\omega)_p=\lambda(p;s)\,\omega_p.
\end{equation}

We want to understand the push-forward of $L$ with respect to the flow of $W$.
Denote by
\[
 \tilde L(p;t):=\Big(T_{\phi_{_{t}}(p)}\phi_{-t}\Big)L(\phi_{_{t}}(p))=a(p;t)\ A(p)+b(p;t)\ B(p)+c(p;t)\ W(p).
\]
Since $L$ is never tangent to $\sW$ we have that $\rho(p;t):=\sqrt{a^2(p;t)+b^2(p;t)}$ is everywhere positive.
Moreover there exists a unique function $\theta(p;t)$ smooth in $t$, such that $\theta(p;0)\in[0,2\pi)$ and
\begin{equation}\label{EQN_L(p,t)}
 \tilde L(p;t)=\rho(p;t)\Big(\cos\theta(p;t)\ A(p)+\sin\theta(p;t)\ B(p)\Big)+c(p;t)\ W(p).
\end{equation}

\subsection{What happens if we change time?}\label{SEC_WhatHappensIfWeChangeTime}
We want to understand how $\theta$ varies if we change $t$.
A straightforward calculation gives (we denote the derivative with respect to $t$ with a dot)
\[
 \D{}{t}\tilde L(p;t)=\frac{\dot\rho(p;t)}{\rho(p;t)}\tilde L(p;t)+\dot\theta(p;t)J_p\tilde L(p;t)\qquad \mod \sW(p).
\]
Since $i_W \omega=0$, we get
\begin{equation}\label{EQN_thetadot}
\begin{split}
 	\rho^2(p;t)\dot\theta(p,t)&=\dot\theta(p,t)\ \omega_p\Big(\tilde L(p;t),\,J_p\tilde L(p;t)\Big)\\
 	&=\omega_p\left(\tilde L(p;t),\,\D{}{t}\tilde L(p;t)\right).
\end{split}
\end{equation}
Using the properties of the flow $\phi_t$ we have
\begin{equation}\label{EQN_L(p;t+s)}
	\begin{split}
		\tilde L(p;t+s)&=\Big(T_{\phi_{t+s}(p)}\phi_{-(t+s)}\Big) \Big(L\big(\phi_{t+s}(p)\big)\Big)\\
		&=\Big(T_{\phi_{s}(p)}\phi_{-s}\Big)\circ\Big(T_{\phi_t(\phi_s(p))}\phi_{-t}\Big) \Big(L\big(\phi_{_{t}}(\phi_s(p))\big)\Big)\\
		&=\Big(T_{\phi_{s}(p)}\phi_{-s}\Big)\Big(\tilde L\big(\phi_s(p),t)\Big).
	\end{split}
\end{equation}
Putting together~\eqref{EQN_alphaUnderTheFlow},~\eqref{EQN_thetadot} and~\eqref{EQN_L(p;t+s)} we get 
\begin{equation*}
\begin{split}
	\rho(p;t+s)^2\,\dot \theta(p;t+s)&=\omega_p\left(\tilde L(p;t+s),\,\D{}{t}\tilde L(p;t+s)\right)\\
	&=\omega_p\left(\Big(T_{\phi_{s}(p)}\phi_{-s}\Big)\tilde L\big(\phi_s(p);t),\,\Big(T_{\phi_{s}(p)}\phi_{-s}\Big)\D{}{t}\tilde L\big(\phi_s(p);t)\right)\\
	&=\Big(\phi^*_s\omega_p\Big)\left(\tilde L(\phi_s(p);t),\,\D{}{t}\tilde L(\phi_s(p);t)\right)\\
	&=\lambda(p;s)\,\omega_p\left(\tilde L(\phi_s(p);t),\,\D{}{t}\tilde L(\phi_s(p);t)\right),\\
\end{split}
\end{equation*}
where in the last step we used the fact that both $\tilde L$ and its derivative are in $\sE$.
Finally we get
\begin{equation}\label{EQN_thetadot(t+s)}
\rho(p;t+s)^2\,\dot \theta(p;t+s)=\lambda(p;s)\,\rho(\phi_s(p);t)^2\,\dot\theta(\phi_s(p);t).
\end{equation}
This equation can be used to prove the following 
\begin{proposition}\label{PROP_EngelConditionOnTheta}
	The distribution $\bra W,\,L\ket\subset\sE$ generates $\sE$ on the orbit of $p$ if and only if $\dot\theta(p;s)\ne 0$ for all $s\in\R$.
\begin{proof}
The distribution generates $\sE$ at $p$ if and only if $\bra W,\,L,\,[W,L]\ket_p=\sE_p$ which happens if and only if $L_p$ and $[W,L]_p$ are linearly independent modulo $\sW_p$.
We write
\[
\begin{split}
	[W,L]_p&=(\sL_WL)_p
	=\restrto{\D{}{t}}{t=0}\Big(T_{\phi_{t}(p)}\phi_{-t}\Big)L\big(\phi_{t}(p))
	=\restrto{\D{}{t}}{t=0}\tilde L(p;t)\\
	&=\dot\rho(p;0) L(p)+\dot\theta(p;0) J_pL(p)\quad\mod{\sW_p},
\end{split}
\]
hence $\bra W,\,L\ket$ generates $\sE$ at $p$ if and only if $\dot\theta(p;0)\ne0$.
Using Equation~\eqref{EQN_thetadot(t+s)} we conclude that
\[
	\dot\theta(\phi_s(p);0)=\frac{\rho(t;s)^2}{\lambda(p;s)}\,\dot\theta(p;s),
\]
which means that $\bra W,\,L\ket$ generates $\sE$ at $\phi_s(p)$ if and only if $\dot\theta(p;s)\ne0$.
\end{proof}
\end{proposition}

\begin{remark}
      Notice that $\dot\theta(p;s)$ is positive (resp. negative) if the orientation of $\sE$ induced by $\bra W,\,L\ket$ is the same as (resp. different from) the one induced by $\bra W,\,A,\,B\ket$.
\end{remark}

\section{Rotation number}\label{SEC_RotationNumber}
We can interpret Proposition~\ref{PROP_EngelConditionOnTheta} geometrically by saying that $\sD=\bra W,\,L\ket$ generates $\sE$ if and only if $L$ ``rotates without stopping around the flow of $W$''.
With this idea in mind we give the following
\begin{definition}
	Let $\sE=\bra W,\,A,\,B\ket$ be as above and let $\gamma\subset M$ be a closed orbit of $W$ of period $T$.
	For a given $p\in\gamma$, we call \emph{rotation number of $L$ around $\gamma$ at $p$} the quantity
	\[
	\rot{L}=\theta(p;T)-\theta(p;0).
	\]
\end{definition}

\begin{remark}\label{REM_RescaleL}
  Notice that in general the rotation number is not an integer and it does not depend on the choice of $\{A(p),\ B(p)\}$.
  Moreover positively rescaling $W$ and $L$ also does not change it.
\end{remark}

The rotation number is not invariant under homotopies of $L$, as the following example shows.

\begin{example}\label{EXM_RotNotInvariant}
	The Lie algebra $\fg$ of the Lie group $\pf{Sol}_1^4$ is generated by $\{W,\,X,\,Y,\,Z\}$ satisfying
	\[
	 [W,X]=-X,\quad[W,Y]=Y,\quad [X,Y]=Z,
	\]
	and all other brackets are zero.
	This implies that we have left-invariant Engel structure given by $\sD=\bra W,\,X+Y\ket$ (see \cite{vogelGeometricEngel} for more details).
	
	Consider the left-invariant even contact structure $\bra W,\,X,\,Y\ket$, by construction the characteristic foliation is spanned by $W$, and its flow preserves $\bra X\ket$ and $\bra Y\ket$.
	This means that for each compact quotient $\pf{Sol}_1^4/\Gamma$ such that $W$ admits a closed orbit $\gamma$, we have $\rot{X}=0$ since $\theta$ is constant.
	On the other hand $L_s=X+sY$ gives a homotopy $\sD_s=\bra W,\,L_s\ket$ between $\sD_0=\bra W,\,X\ket$ and $\sD_1=\bra W,\,X+Y\ket$, which is an Engel structure.
	In particular $\rot{X+Y}\ne0$ by Proposition~\ref{PROP_EngelConditionOnTheta}.
\end{example}

We have invariance under a smaller family of homotopies of $L$.
\begin{lemma}\label{LEM_HomInvarianceRot}
	Let $L_\tau$ for $\tau\in[0,1]$ be a smooth family of vector fields tangent to $\sE=\bra W,\,A,\,B\ket$ and nowhere tangent to $\sW$.
	If $L_\tau(p)=L_0(p)$ for all $\tau\in[0,1]$ then 
	\[
		\rot{L_1}=\rot{L_0}.
	\]
	
\begin{proof}
	Using~\eqref{EQN_L(p;t+s)} with $t=0$ and $s=T$ we get
	\[
		\tilde L_\tau(p;T)=\Big(T_{\phi_{_{T}}(p)}\phi_{-T}\Big)\tilde L_\tau\big(\phi_{_{T}}(p);0)=\Big(T_p\phi_{-T}\Big) L_\tau(p).
	\]
	By a calculation similar to the one performed in~\eqref{EQN_thetadot(t+s)}, we get
	\[
	\begin{split}
	 	\rho_\tau(p;T)^2\D{}{\tau}\theta_\tau(p;T)&=
	 	\omega_p\left(\tilde L_\tau(p;T),\,\D{}{\tau}\tilde L_\tau(p;T)\right)\\
	 	&=\phi_{_{T}}^*\omega_p\left(L_\tau(p),\,\D{}{\tau}L_\tau(p)\right)\\
	 	&=\lambda(p;T)\D{}{\tau}\theta_\tau(p;0)=0.
	\end{split}
	\]
	This readily implies that $\rot{L_\tau}$ is constant in $\tau$.
\end{proof}
\end{lemma}

The previous result suggests to take into account all possible ``initial phases''.
More precisely recall that $L$ tangent to $\bra A,\,B\ket$ can be identified with a map $L:M\to S^1$.
Now take $\eta\in\R$ and denote by $R(\eta)$ the rotation  of $S^1$ of angle $\eta$.
The idea is to consider the rotation number of all vector fields obtained via $R(\eta)\circ L$ for $\eta\in\R$.
We define
\begin{equation}
	\Phi^L_{\gamma,p}:\R\to\R\spacedt{s.t.}\eta\mapsto \rot{R(\eta)\circ L}.
\end{equation}
The \emph{maximal rotation number of $L$ along $\gamma$ at $p$} is
\[
	\maxrot=\max_\eta\Big(\Phi^L_{\gamma,p}(\eta)\Big).
\]
This object detects whether $\sD_L=\bra W,\,L\ket$ generates $\sE$ on $\gamma$ up to homotopy within $\sE$.

\begin{theorem}\label{THM_IFFforHomotMaxNonIntegrable}
	Let $\sE=\bra W,\,A,\,B\ket$ be a distribution of rank $3$ such that $[\sW,\sE]\subset\sE$, and let $\gamma$ a closed orbit for $W$.
	Then $\sD_L=\bra W,\,L\ket$ generates $\sE$ on $\gamma$ up to homotopy within $\sE$ if and only if there exists a point $p\in\gamma$ such that $|\maxrot|>0$.

\begin{proof}
	 Suppose first that $\sD_L=\bra W,\,L\ket$ generates $\sE$ up to homotopy on $\gamma$, and let $L_\tau$ for $\tau\in[0,1]$ be a homotopy such that $\bra W,\,L_1\ket$ is maximally non-integrable within $\sE$.
	 We need to show that for a given $p\in\gamma$, there is a homotopy relative to $L_1(p)$ between $L_1$ and $R(\eta)\circ L$ for some $\eta\in\R$.
	 This implies indeed by Lemma~\ref{LEM_HomInvarianceRot} and Proposition~\ref{PROP_EngelConditionOnTheta} that 
	 \[
	    \Phi^L_{\gamma,p}(\eta)=\rot{L_1}\ne0.
	 \]
	 For a fixed $p\in M$ there is an angle $\eta$ such that $R(\eta)\circ L(p)=L_1(p)$.
	 On $\Op(p)$ we can homotope $R(\eta)\circ L$ to $L_1$ relative to $\{p\}$ (here we possibly need to rescale the vector, as in Remark~\ref{REM_RescaleL}).
	 Since both $R(\eta)\circ L$ and $L_1$ are homotopic to $L$, they must be homotopic to each other, and since the fundamental group of $S^1$ is Abelian, there is a homotopy relative to $\{p\}$.

 	Conversely suppose $|\maxrot|>0$.
 	Without loss of generality we can suppose that $\rot{L}>0$.
 	First consider a small neighbourhood of $p$ and homotope $L$ relative to $\{p\}$ and to the boundary of $\Op(p)$ to a maximally non-integrable distribution within $\sE$ near $p$.
 	Since this homotopy is relative to $\{p\}$, the rotation number does not change by Lemma~\ref{LEM_HomInvarianceRot}.
 	
 	For $\eps>0$ small, take a disc $D^3\imm M$ centered at $\phi_\eps(p)$ and everywhere transverse to $\sW$.
 	Up to shrinking the disc $D^3$, we can suppose that the map $F:D^3\times[\eps,T-\eps]\to M$ given by the flow $(q,t)\mapsto\phi_t(q)$ is an embedding and hence a flow box for $W$.
 	In this chart we have
 	\[
 		F^*\tilde L(q;t)=\rho(q;t)\Big(\cos\psi(q;t)\ F^*A(p)+\sin\psi(q;t)\ F^*B(p)\Big)
 	\]
 	with $\psi(0;t)=\theta(\phi_\eps(p);t)$ on $\gamma$.
 	Since $\rot{L}>0$, up to choosing $\eps>0$ small enough, we can suppose that $\psi(0;T-\eps)-\psi(0;\eps)>0$.
 	Hence there exists a homotopy $\psi_\tau:D^3\times[\eps,T-\eps]\to\R$ such that $\psi_0=\psi$, the restriction of $\psi_\tau$ to the boundary $\partial(D^3\times[\eps,T-\eps])$ is $\psi$, and 
 	\[
 		\psi_1(0;t)=h(t)\Big(\psi(0;T-\eps)-\psi(0;\eps)\Big)+\psi(0;\eps)
	\]
	for a smooth step function $h$ (see Figure~\ref{FIG_HomotTheta}).
 	Then $\sD_{L_1}$ generates $\sE$ on a (possibly smaller) neighbourhood of $\gamma$.
\end{proof}
\end{theorem}

\begin{figure}
      \centering
	\includegraphics[width=\columnwidth]{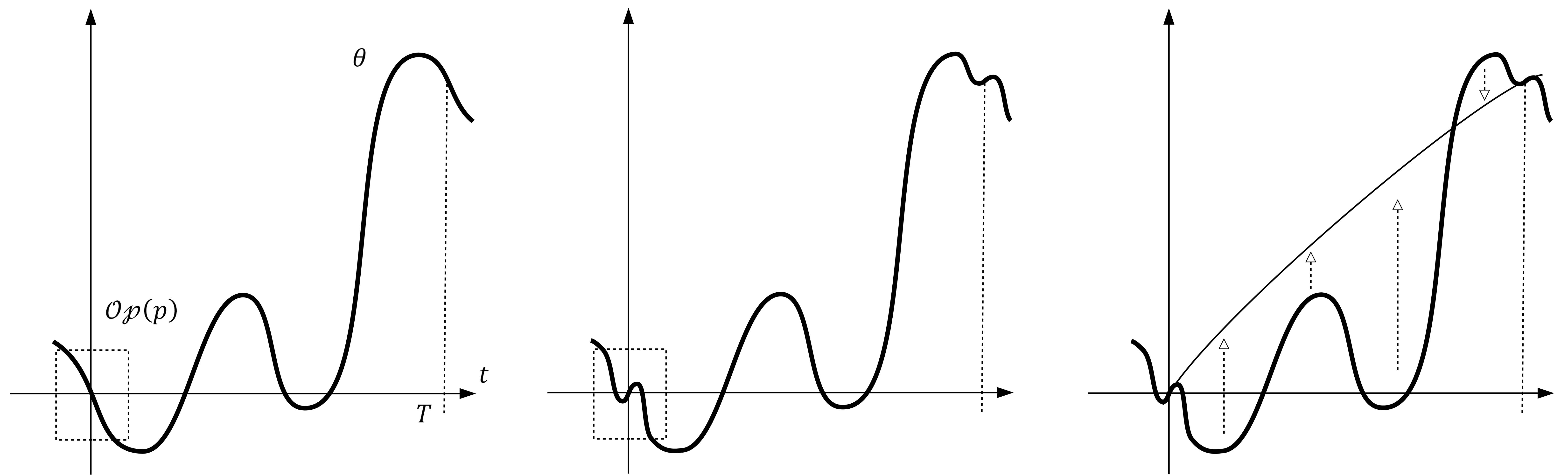}
	\caption{Homotopy of $\theta$ when the rotation number is positive.}
	\label{FIG_HomotTheta}
\end{figure}

\section{Character of closed orbits of $\sW$}\label{SEC_CharacterOfWOrbits}

We now consider the action of the flow of $W$ on $\sE/\sW$.
This is discussed in details in \cite{mitsu} for the case of Engel structures.

For a fixed section $\sW=\bra W\ket$, we have the action of the associated flow $\phi_t$ on $\sE$ and this induces a map $[T\phi_t]:\PP(\sE_p/\sW_p)\to\PP(\sE_{\phi_t(p)}/\sW_{\phi_t(p)})$.
If $p\in M$ is contained in a closed orbit of $\sW$ of period $T$, then $P:=[T_p\phi_T]\in\pf{PSL}(2,\R)$, where we identify $\R\PP^1=\PP(\sE_p/\sW_p)$.

Hence the usual classification of such maps we say that a closed orbit $\gamma$ is
\begin{itemize}
	\item \emph{Elliptic} if $|\pf{tr}{P}|<2$ or $P=\pm id$, in this case we can represent $P$ by a matrix of the form
	\[
	P\equiv\pm\Matrix{\cos\delta&\sin\delta\\-\sin\delta&\cos\delta}
	\]
	for some $\delta\in\R$.
	\item \emph{Parabolic} if $|\pf{tr}{P}|=2$ and $P\ne\pm id$, in this case we can represent $P$ by a matrix of the form
	\[
	P\equiv\pm\Matrix{1&\pm1\\0&1}.
	\]
	\item \emph{Hyperbolic} if $|\pf{tr}{P}|>2$, in this case we can represent $P$ by a matrix of the form
	\[
	P\equiv\pm\Matrix{e^{-\mu}&0\\0&e^{\mu}}
	\]
	for some $\mu\in\R$.
\end{itemize}
Notice that parabolic orbits come in two types depending on the sign of the entry over the diagonal.
We will call \emph{positive parabolic} orbits where this element has sign opposite to the one of the diagonal entries, and negative otherwise.

The following result analyses how $\Phi^L_{\gamma,p}(\eta)$ changes when we change $\eta$.
\begin{proposition}\label{PROP_RotForTypeOfClosedOrbits}
    Let $\sE=\bra W,\,A,\,B\ket$ be a distribution of rank $3$ such that $[\sW,\sE]\subset\sE$, let $\gamma$ be a closed orbit for $W$, and let $p$ be a point on $\gamma$.
    \begin{enumerate}
    	\item If $\gamma$ is hyperbolic, then for every $\eta\in\R$ we have 
    	\[
    		\left|\Phi^L_{\gamma,p}(\eta)-\rot{L}\right|<\frac{\pi}{2}.
    	\]
    	Moreover there exists a constant $c\in(0,\, \pi/2)$ such that $\sD_L=\bra W,\,L\ket$ positively generates $\sE$ on $\gamma$ up to homotopy if and only if $\rot{L}>-c$.
    	
    	\item If $\gamma$ is positive parabolic, then for every $\eta\in\R$ we have 
    	\[
    	\left|\Phi^L_{\gamma,p}(\eta)-\rot{L}\right|<\pi.
    	\]
    	Moreover there exists a constant $c\in(0,\, \pi)$ such that $\sD_L=\bra W,\,L\ket$ positively generates $\sE$ on $\gamma$ up to homotopy if and only if $\rot{L}>-c$.
    	
    	\item If $\gamma$ is negative parabolic then $\sD_L=\bra W,\,L\ket$ positively generates $\sE$ on $\gamma$ up to homotopy if and only if $\rot{L}>0$.
    	
    	\item If $\gamma$ is elliptic then $\Phi^L_{\gamma,p}(\eta)$ is constant in $\eta$.
	\end{enumerate}
    \begin{proof}
      We use the notation
      \[
      \Big(T_{\phi_t(p)}\phi_{-t}\Big)\Big(R(\eta)\circ L(\phi_t(p))\Big)=\rho_\eta(p;t)\Big(\cos\theta_\eta(p;t)\ A(p)+\sin\theta_\eta(p;t)\ B(p)\Big)
      \]
      modulo $\sW(p)$ where $\theta_\eta(p;0)=\theta(p;0)+\eta$.
      We need to determine $\theta_\eta(p;T)$.
      Set
      \[
    		M_\eta(t):=\Big(T_{\phi_t(p)}\phi_{-t}\Big)R(\eta) \Big(T_{\phi_t(p)}\phi_{-t}\Big)^{-1},
    	\]
		so that
      \begin{align*}
            \tilde L_\eta(p;T)&=\Big(T_{p}\phi\su{-T}\Big)\Big(R(\eta)\circ L(p)\Big)=M_\eta(T)\Big(T_{p}\phi\su{-T}(L(p))\Big)\\
            &=\rho(p;T)M_\eta(T)\Big(\cos\theta(p;T)\,A(p)+\sin\theta(p;T)\,B(p)\Big)
      \end{align*}
      modulo $\sW(p)$.
      There is a function $r=r(\eta,\theta)$ which depends on $M_\eta(T)$ and on the angle $\theta(p;T)$ such that
      \[
      \tilde L_\eta(p;T)=\tilde\rho(p;T)\Big(\cos(\theta(p;T)+r)\ A(p)+\sin(\theta(p;T)+r)\ B(p)\Big)
      \]
      modulo $\sW(p)$.
      Moreover since $M_\eta(0)=R(\eta)$ and $M_\eta(t)$ is continuous we conclude that $\theta_\eta(p;T)=\theta(p;T)+r$.
      We will discuss what the angular displacement $r$ is in the various cases.
      
      Since the only term in Equation~\ref{EQN_L(p,t)} that is used to calculate $\rot{L}$ is $\theta$, it suffices to consider the conformal class of the map $T_p\phi_t$.
      Take $P=\lambda\, T_{p}\phi\su{-T}$ where $\lambda$ is a constant such that $\det P=1$.
      
      If $\gamma$ is hyperbolic, using Remark~\ref{REM_RescaleL} we can choose $A(p)$ and $B(p)$ so that $P(A(p))=e^{-\mu}\, A(p)$ and $P(B(p))=e^\mu\, B(p)$, so that in this basis
      \[
      	P=\Matrix{e^{-\mu}&0\\0&e^{\mu}}.
      \]
      Moreover up to changing initial phase $L(p)=A(p)$.
      This means that $r=\eta+\eps$ where $|\eps|<\pi/2$, since $P$ displaces a point $q\in S^1$ of an angle of at most $\pi/2$.
      Now
      \[
      	\Phi^L_{\gamma,p}(\eta)=\theta_\eta(p;T)-\theta_\eta(p;0)=\theta(p;T)+\eta+\eps-\theta(p;0)-\eta=\rot{L}+\eps
      \]
      and it suffices to set $c=|\min_\eta(\eps(\eta))|$.
      
      If $\gamma$ is positive parabolic, we can suppose as above that $L(p)=A(p)$ and that $P$ in the basis $\{A(p),\,B(p)\}$ takes the form
      \[
      	P=\Matrix{1&-1\\0&1}.
      \]
      As above $r=\eta+\eps$ with $0\le\eps<\pi$ and again we set $c=|\min_\eta(\eps(\eta))|$.
      On the other hand if $\gamma$ is negative parabolic then by changing $\eta$ we can only decrease the rotation number, i.e. $r=\eta-\eps$ with $\eps$ as above.
      
      Finally if $\gamma$ is elliptic for any choice of $\{A(p),\,B(p)\}$ we have
      \[
      	P=\Matrix{\cos\delta&\sin\delta\\-\sin\delta&\cos\delta}=R(\delta),
      \]
      so that $M_\eta(T)=R(\delta) R(\eta) R(\delta)^{-1}=R(\eta)$ and $\theta_\eta(p;T)=\theta(p;T)+\eta$.
      
    \end{proof}
\end{proposition}

\begin{remark}
	The previous result ensures that the only cases where it can happen that $\rot{L}\le0$ and nonetheless $\sD_L=\bra W,\,L\ket$ positively generates $\sE$ up to homotopy on $\gamma$ occur when $\gamma$ is hyperbolic or positive parabolic.
	Moreover in these cases $\rot{L}$ is not allowed to be ``too negative''.
	In order to overcome this subtlety we will consider the quantity $\maxrot$ in what follows.
\end{remark}

The previous proposition gives crucial information in the case of an unknotted closed orbit.
\begin{corollary}\label{COR_ContractibleAreObstructionForEngel}
	Let $\sE=\bra W,\,A,\,B\ket$ be a distribution of rank $3$ such that $[\sW,\sE]\subset\sE$, and let $\gamma$ be an unknotted elliptic closed orbit for $W$.
	Then the rotation number $r=\rot{L}$ of $L\in\Gamma\bra A,\, B\ket$ does not depend on $L$.
	In particular there exists an oriented plane field $\sD$ such that $\sW\subset\sD\subset\sE$ which positively generates $\sE$ on a neighbourhood of $\gamma$ if and only if $r>0$.
	
	\begin{proof}
		Let $L$ and $L'$ be non-singular vector fields in $\bra A,\, B\ket$, recall that these can be identified with maps $L,L':M\to S^1$.
		Since $\gamma$ is unknotted, there is an embedded disc $D^2$ such that $\partial D^2=\gamma$, hence there exists a homotopy between $L$ and $L'$.
		Since $\gamma$ is elliptic, by point (4) of Proposition~\ref{PROP_RotForTypeOfClosedOrbits} we have that $r=\rot{L}=\rot{L'}$.
		The second claim follows now directly from Theorem~\ref{THM_IFFforHomotMaxNonIntegrable}.
	\end{proof}
\end{corollary}

The previous corollary gives a necessary condition for a distribution of rank $3$ to admit a maximally non-integrable plane field.
Notice that the hypothesis that $\gamma$ is unknotted is equivalent to the fact that $\gamma$ is null-homotopic if the dimension of $M$ is greater than $3$.

\section{Morse-Smale vector fields}\label{SEC_MSVF}
The goal of this section is to apply the machinery developed in the previous sections to the case where $W$ is a non-singular Morse-Smale vector field.
Since the dynamics of these vector fields can be described once we understand neighbourhoods of the closed orbits, it is reasonable to expect that the rotation number will play a central role in this context.
We now recall some basic facts from the theory of Morse-Smale vector fields, see~\cite{irwin,katok,morgan} for more details.

\subsection{Morse-Smale vector fields and round handle decompositions}
A \emph{non-singular Morse-Smale vector field (NMS)} $W$ on a manifold $M$ is a non-singular vector field which satisfies the following conditions
\begin{enumerate}
	\item $W$ has finitely many closed orbits $\gamma_1,...,\gamma_k$ and they are all non-degenerate;
	\item the non-wandering set is the union of the closed orbits $\Omega=\gamma_1\cup\cdots\cup\gamma_k$;
	\item for every $i,j\in\{1,...,k\}$ the stable manifold $W^s(\gamma_i)$ and the unstable manifold $W^u(\gamma_j)$ intersect transversely.
\end{enumerate} 

The main reason why we are interested in Morse-Smale vector fields is that their dynamical properties are, in a sense, completely determined by what happens near the closed orbits and by how these are linked together.
Recall that a round handle decomposition (RHD) of $M$ is a decomposition of $M$ in pieces of the form $R_k=D^k\times D^{n-k-1}\times S^1$ called round handles.
We call $\partial_+R_k=D^k\times\partial D^{n-k-1}\times S^1=D^k\times S^{n-k-2}\times S^1$ the enter region or the positive boundary and $\partial_-R_k=\partial D^k\times D^{n-k-1}\times S^1=S^{k-1}\times D^{n-k-1}\times S^1$ the exit region or the negative boundary.
\begin{theorem}[\cite{morgan}]\label{THM_DecompositionAssociatedToNSMS}
	Let $W$ be a non-singular Morse-Smale vector field on $M$.
	Then $M$ admits a RHD $M_0\subset M_1\subset \cdots \subset M_k=M$ such that every round handle $R$ is a neighbourhood of closed orbits $\gamma$ of $W$ and the index of $R$ (as a handle) is the index of $\gamma$ (as a closed orbit).
	Moreover the attaching procedure is performed using the flow of $W$, which is transverse to every $M_i$.
\end{theorem}

The idea of the proof is to order the closed orbits of $M$ via $\gamma_i\le\gamma_j$ if $W^u(\gamma_i)\cap W^s(\gamma_j)\ne0$, and reason by induction.
Otherwise said $\gamma_i\le\gamma_j$ if there is a orbit whose $\alpha$-limit is $\gamma_i$ and whose $\omega$-limit is $\gamma_j$.
The following result ensures that this is compatible with a total ordering of $\{\gamma_1,...,\gamma_k\}$.
\begin{theorem}[\cite{smale} (No cycle condition)]
	Let $\{\gamma_1,...,\gamma_k\}$ be the set of closed orbits of a non-singular Morse-Smale vector field with the ordering defined above.
	Then there exists no non-trivial sequence $\gamma_{i_1}\le\gamma_{i_2}\le\cdots\le\gamma_{i_1}$.
\end{theorem}

In order to construct the RHD of $M$ one starts by attaching the source orbits by disjoint union, then a generic point in $M\setminus\{\gamma_1,...,\gamma_k\}$ has to have one of the source orbits as $\omega$-limit.
Suppose that we have constructed inductively $M_i$ such that
\begin{itemize}
	\item $\gamma_1,...,\gamma_i\in M_i$;
	\item $\gamma_j\cap M_i=\emptyset$ for $j>i$;
	\item the flow is transverse pointing outward on $\partial M_i$.
\end{itemize}
We take a small tubular neighbourhood $R_{i+1}$ of $\gamma_{i+1}$, the construction of the ordering ensures that points in $R_{i+1}\setminus{\gamma_i}$ have $\omega$-limit in $M_i$.
We attach $R_{i+1}$ using all flow lines of $W$ that have $\omega$-limit in $M_i$.
The problem with this procedure is that it may introduce corners.
Moreover the boundary of $M_{i+1}$ will not be transverse to $W$.
The solution is to smoothen the corners as illustrated in Figure~\ref{FIG_RoundingCorners}.
\begin{figure}
	\centering
	\includegraphics[width=0.8\columnwidth]{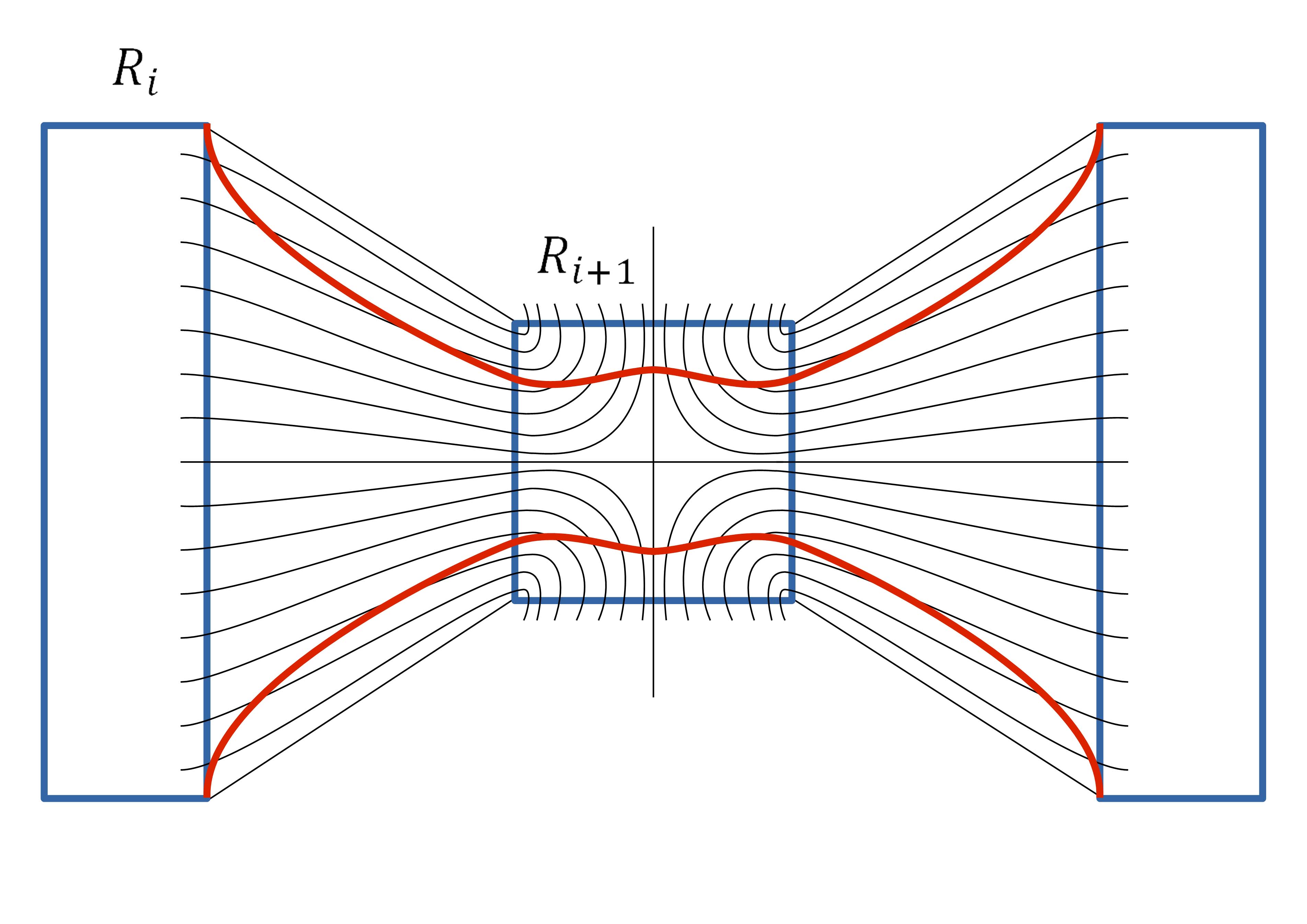}
	\caption{Smoothen the corners.}
	\label{FIG_RoundingCorners}
\end{figure}

For further details on the proof see~\cite{morgan}, and for the basic theory of Morse-Smale vector fields see~\cite{irwin, katok}.

\subsection{Morse-Smale flows preserving a $3$-distribution}
Let $\sE=\bra W,\,A,\,B\ket$ be a rank $3$ distribution such that $[\sW,\sE]\subset\sE$ and suppose that $W$ is NMS.
The following result furnishes a necessary and sufficient condition for the existence of $\sD\subset\sE$ that generates $\sE$.
\begin{theorem}\label{THM_ExistenceOfDForNMS}
    Let $\sE=\bra W,\,A,\,B\ket$ be a rank $3$ distribution such that $[\sW,\sE]\subset\sE$, and let $W$ be a NMS vector field.
    There exists $\sD\subset\sE$ that positively generates $\sE$ if and only if there exists a vector field $L\in\bra A,\,B\ket$ such that $\maxrot>0$ for all $\gamma$ closed orbit of $W$.
\end{theorem}

The idea of the proof is to first construct $\sD$ on a neighbourhood of the closed orbits and then attach the handles using the flow of $W$.
Suppose that we have constructed $\sD$ on $M_{k-1}$ and we want to attach a new round handle $R_k$.
Using Theorem~\ref{THM_IFFforHomotMaxNonIntegrable} we get a plane field $\sD$ on $M_k=M_{k-1}\cup R_k$ which generates $\sE$ on a neighbourhood of $M_{k-1}$ and on a neighbourhood of the core of $R_k$.
On the attaching region we have $L=\cos{f_t}\, A+\sin{f_t}\, B$ for a family of angle functions $f_t$.
Now $\sD$ is homotopic to a plane field generating $\sE$ on the whole $M_k$ only if $f_0<f_1$.
This will not happen in general.

The idea to overcome this is that we are not interested in plane fields homotopic to $\bra W,\,L\ket$, we just need a plane field which positively generates $\sE$.
Hence instead of glueing $R_k$ directly, we first make sure to increase $f_1(p)$ using the fact that $W$ is transverse to $\partial M_{k-1}$.
We take a collared neighbourhood of the boundary where $L$ rotates positively ``a bit'' and we substitute it with one where $L$ rotates ``massively''.
In this way the condition $f_0<f_1$ is verified for the new field, which nonetheless coincides with $L$ in a neighbourhood of the closed orbits of $W$.

Once the glueing is done we use the procedure described in the proof of Theorem~\ref{THM_DecompositionAssociatedToNSMS} to smoothen the corners.
Indeed this ``digs'' the new $M_k$ inside the manifold with corners that we have just constructed.
After this process the plane field $\sD$ will be the restriction of the previously constructed one.

\begin{proof}[Proof of Theorem~\ref{THM_ExistenceOfDForNMS}]
	If such a plane field $\sD$ exists then we can take $L$ to be any vector $\sD=\bra W,\,L\ket$ and the claim follows by Proposition~\ref{PROP_EngelConditionOnTheta}.
	Conversely suppose that $L$ satisfies the above properties.
	The idea is to construct $\sD$ inductively using the decomposition provided by Theorem~\ref{THM_DecompositionAssociatedToNSMS}.
	
	The first step of the induction is to construct $\sD$ in the neighbourhood of the sources.
	This is possible thanks to Theorem~\ref{THM_IFFforHomotMaxNonIntegrable}.
	This procedure yields a plane field $\sD$ homotopic to $\bra W,\,L\ket$ which generates $\sE$ on the handle.
	Notice that we can make sure that the boundary of the (possibly disconnected) manifold that we obtain with this procedure is transverse to $W$.
	
	For the inductive step suppose that we have attached $k-1$ handles to obtain $M_{k-1}$, and that we want to attach the $k$-th handle $R_k$.
	Theorem~\ref{THM_DecompositionAssociatedToNSMS} ensures that $R_k$ is a neighbourhood of $\gamma_k$, and that the attaching procedure happens via the flow of $W$.
	We first construct $\sD$ on $R_k$ using Theorem~\ref{THM_IFFforHomotMaxNonIntegrable}, this is possible because of the hypothesis on $L$.
	The existence of $L$ also ensures that the $\sD$ on $M_{k-1}$ extends to a plane field on $M_k$ which generates $\sE$ on a neighbourhood of $M_{k-1}$ and of $\gamma_k$.
	
	In general we cannot homotope this plane field to a maximally non-integrable one on $M_k$.
	The problem is that the attaching region is of the form $R_k^+\times I$, where $R_k^+\times\{1\}$ is the subset of $R_k$ where $W$ points inwards, and $W$ is tangent to the $I$ factor on $R_k^+\times I$.
	This means that the restriction of $L$ to $\tilde{R}_k^+\times I$ takes the form $L=\cos{f_t}\, A+\sin{f_t}\, B$, where $f_t:{\tilde R}_k^+\times\{t\}\to \R$ is a $I$-family of angle functions.
	Hence we can homotope $L$ transversely to $\partial_t$ so that $\bra \partial_t,\,L\ket$ generates $\sE$ if and only if $f_1>f_0$.
	There is no reason for this to happen in general.
	
	Let $K=\max\{f_1(p)-f_0(p)|\,p\in {\tilde R}_k\}$.
	For any $p\in\partial M_{k-1}\times(-\eps,\eps)$ the vector field $L$ can be described by a map $h_p:(-\eps,\eps)\to S^1$ which is a small embedding.
	We substitute it with $\tilde h_p:(-\eps,\eps)\to S^1$ which coincides with $h$ on $\Op(\{-\eps,\eps\})$, and such that it makes a number of turns around $S^1$ bigger than $K$.
	The net effect of this is that the difference between the new angle functions $f_1$ and $f_0$ is positive.
	This ensures that we can homotope $L$ to a maximally non-integrable plane field within $\sE$ on the attaching region.
	
	We might now need to round the corners of $M_k$, and this can be done exactly as in the proof of Theorem~\ref{THM_DecompositionAssociatedToNSMS} (see Figure~\ref{FIG_RoundingCorners}).
\end{proof}

\section{Morse-Smale Legendrian vector fields}\label{SEC_MSLegendrian}
In this section we apply the above discussion to the case of a manifold $M$ of dimension $3$ and $\sE=TM$.
In this context a plane field $\sD\subset\sE$ which generates $\sE$ is a contact structure, so that Theorem~\ref{THM_IFFforHomotMaxNonIntegrable} gives immediately the following
\begin{corollary}
  Let $M$ be a closed orientable $3$-manifold, $W$ a non-singular vector field on $M$ such that $TM=\bra W,\,A,\,B\ket$, and $\gamma$ a closed orbit of $W$.
  The plane field $\sD_L=\bra W,\,L\ket$ is a contact structure on $\gamma$ up to homotopy\footnote{See Remark~\ref{REM_Homotopies}} if and only if $|\maxrot|\ne0$.
\end{corollary}

Theorem~\ref{THM_ExistenceOfDForNMS} gives the following
\begin{corollary}
    Let $M$ be a closed orientable $3$-manifold and let $W$ be a NMS vector field on $M$ such that $TM=\bra W,\,A,\,B\ket$.
    There exists a positive contact structure $\sD$ for which $W$ is Legendrian if and only if there exists a vector field $L\in\bra A,\,B\ket$ such that $\maxrot>0$ for all $\gamma$ closed orbit of $W$.
\end{corollary}

An interesting example of $3$-manifold admitting NMS vector fields is $S^3$.
On the other hand only very few $3$-manifolds admit such vector fields.
The following result classifies completely the topology of such manifolds.
\begin{theorem}[\cite{morgan}]
	Let $M$ be an orientable, prime $3$-manifold with boundary such that the Euler characteristic of every boundary component vanishes.
	Let $\partial_-M$ be an arbitrary union of these components.
	Suppose $M$ is not $S^1\times D^2$.
	The pair $(M,\partial_-M)$ admits a non-singular Morse-Smale flow if and only if $M$ is a union of non-trivial Seifert spaces attached to one another along components of their boundaries.
\end{theorem}

\begin{remark}
It is interesting to know when a given vector field $L\in\fX(M)$ is transverse to a contact structure.
This question was already studied in \cite{girouxMW} for the case where $L$ is tangent to the fibres of a $S^1$-bundle over a surface, and in \cite{lisca} for the case $L$ tangent to the fibres of a Seifert fibration.

Notice that if $L$ is Legendrian for some orientable contact structure $\sD$ then there is a contact structure $\tilde\sD$ transverse to $L$.
Indeed choose $\tilde L$ such that $\sD=\bra L,\,\tilde L\ket$ and consider $\tilde\sD=\phi_*\sD$, where $\phi$ denotes the flow of $\tilde L$ for small time.
The contact condition ensures that $\tilde\sD$ is transverse to $\sD$, moreover it contains $\tilde L$, so it is transverse to $L$.

 With the techniques developed in this paper we can present an example of a vector field which is transverse to a contact structure but never Legendrian.
 Namely consider the field $W$ normal to the canonical Reeb foliation on $S^3$.
 This can be obtained by glueing a sink and source with trivial monodromy on $S^1\times D^2$ via the $(1,1)$-map.

 Using the theory of confoliations \cite[Chapter 2]{confoliations} we can $\sC^0$-deform the tangent bundle of the Reeb foliation to get a contact structure, so that $L$ is transverse to a contact structure.
 On the other hand $L$ has two unknotted closed orbits which have trivial monodromy, which means that they are elliptic and have rotation number $0$.
 This obstructs the existence of a contact structure for which $L$ is Legendrian.
\end{remark}

\section{Morse-Smale even contact structures}\label{SEC_MSEngel}
We now turn the attention to the case where $\dim M=4$ and $\sE$ is an even contact structure with characteristic foliation $W$.
Theorem~\ref{THM_IFFforHomotMaxNonIntegrable} gives immediately
\begin{corollary}
  Let $M$ be a closed orientable $4$-manifold, $\sE=\bra W,\,A,\,B\ket\subset TM$ an even contact structure with characteristic foliation spanned by $W$, and $\gamma$ a closed orbit of $W$.
  The plane field $\sD_L=\bra W,\,L\ket\subset\sE$ is an Engel structure on $\gamma$ up to homotopy\footnote{See Remark~\ref{REM_Homotopies}} if and only if $|\maxrot|\ne0$.
\end{corollary}

We say that an even contact structure is Morse-Smale, if its characteristic foliation admits a section $W$ which is a NMS vector field.
Theorem~\ref{THM_ExistenceOfDForNMS} gives the following
\begin{corollary}
	Let $(M,\,\sE=\bra W,\,A,\,B\ket)$ be a closed, oriented Morse-Smale even contact $4$-manifold.
	Then there exists a positive Engel structure $\sD$ compatible with $\sE$ if and only if there exists a vector field $L\in\bra A,\,B\ket$ such that $\maxrot>0$ for all $\gamma$ closed characteristic orbits.
\end{corollary}

It is not clear if every parallelizable manifold $4$-manifold admits a Morse-Smale even contact structure.
In fact many NMS flows on $4$-manifolds do not have the right monodromy for being the characteristic foliation of an even contact structure.
On the other hand if we allow $\sC^0$-perturbations of $W$ then we can always suppose that the closed orbits have tubular neighbourhoods $\nu \gamma=S^1\times D^3$ where $W$ writes as
\[
\restrto{W}{\nu\gamma}=\partial_{\theta}+2\eps_1x\partial_{x}+2\eps_2y\partial_{y}+4\eps_3z\partial_{z},
\]
where $\eps_i=\pm1$ depending on the index of $\gamma$.
These models always permit to construct an even contact form $\alpha$ on $\nu\gamma$ such that $W$ spans the characteristic foliation.
If the $\eps_i$ are all equal then $W$ is Liouville for (a multiple of) the symplectic form $\omega=dx\wedge dy+dz\wedge d\theta$, so that we take $\alpha=i_W\omega$.
If the $\eps_i$ are not all equal, then the vector field $V=2\eps_1x\partial_{x}+2\eps_2y\partial_{y}+4\eps_3z\partial_{z}$ preserves the contact structure defined by $\eta=dz-xdy+ydx$ on $D^3$, so that $\nu\gamma$ can be seen as the suspension of the time $1$ flow of $V$.

\begin{example}
	Morse-Smale even contact structures can be obtained by suspension of a Morse-Smale contactomorphism whose non-wandering set only consists of fixed points.
	A way of constructing such contactomorphism is to look for (singular) contact vector fields which are Morse-Smale and do not have closed orbits.
	
	An explicit example is given by the contact vector field $V$ on $(S^3,\,\alpha_{st})$ associated with the contact Hamiltonian $h(x_1,y_1,x_2,y_2)=y_1/2$.
	One can verify that $V$ take the form
	\[
		V(x_1,y_1,x_2,y_2)=\frac{1}{2}\left((1-x_1^2)\,\partial_{x_1}+y_1\,\partial_{y_1}+x_2\,\partial_{x_2}+y_2\,\partial_{y_2}\right).
	\]
	We get an even contact structure on $M=S^3\times S^1$ whose characteristic foliation only has $2$ closed orbits, namely a source and a sink.
	This is induced by an Engel structure since $M$ is obtained as a suspension of a contactomorphism isotopic to the identity (see \cite{mitsu}).
\end{example}

The methods developed in this paper are well-suited for constructing examples of even contact structures which do not admit compatible Engel structures.
Indeed it suffices to construct locally closed orbits of $\sW$ which are null-homotopic and have trivial monodromy and then extend these to the whole manifold via the complete h-principle in \cite{mcduff}.

\end{document}